\documentclass[a4paper]{article}%

\usepackage{a4wide}
\usepackage{amsmath}
\usepackage{amsfonts}
\usepackage{amssymb}
\usepackage{authblk}
\usepackage{amsthm}
\usepackage{bbm}
\usepackage{graphicx}
\usepackage{caption}
\usepackage{mathrsfs}
\usepackage{enumitem}
\usepackage{subcaption}
\usepackage{xspace}
\usepackage{booktabs}
\usepackage[english]{babel}% Engelse afbreekregels
\usepackage{tikz}
\usepackage{marvosym}
\usepackage{cite}
\usepackage{hyperref} % voor hyperlinks verwijzingen
\usepackage{multirow,tabularx}
\usepackage{relsize}

\newcommand\abs[1]{\left|#1\right|}

\newcommand{\e}{{\rm e}}
\newcommand{\w}{h}
\newcommand{\dd}{{\rm d}}
\newcommand{\mean}[1]{\langle #1 \rangle}

\newtheorem{proposition}{Proposition}

\newtheorem{lemma}{Lemma}
\newtheorem{remark}{Remark}

\begin{document}

\title{Local clustering in scale-free networks with hidden variables}% Force line breaks with \\

 \author{Remco van der Hofstad}

\author{A.J.E.M. Janssen}
\author{Johan S.H. van Leeuwaarden}%
 \author{Clara Stegehuis}

\affil{Eindhoven University of Technology
%\\
}%

%\affiliation{

\date{\today}% It is always \today, today,fare the pdf's as

%\pacs{89.75.-k Complex systems, 64.60.aq Networks}
                             % \graphicspath{{New_Figures/}}
\maketitle
\begin{abstract}
	We investigate the presence of triangles in a class of correlated random graphs in which hidden variables determine the pairwise connections between vertices. The class rules out self-loops and multiple edges and allows for negative degree correlations (disassortative mixing) due to infinite-variance degrees controlled by a structural cutoff $h_s$ and natural cutoff $h_c$. We show that local clustering decreases with the hidden variable (or degree). We also determine how the average clustering coefficient $C$ scales with the network size $N$, as a function of $h_s$ and $h_c$. For scale-free networks with exponent $2<\tau<3$ and the default choices $h_s\sim N^{1/2}$ and $h_c\sim N^{1/(\tau-1)}$ this gives $C\sim N^{2-\tau}\ln N$ for the universality class at hand. We characterize the extremely slow decay of $C$ when $\tau\approx 2$ and show that for $\tau=2.1$, say, clustering only starts to vanish for networks as large as $N=10^{9}$.
\end{abstract}

\section{Introduction}
Random graphs serve as models for large networked systems that arise in nature or in our technosphere.
The shear complexity of many such networks prevents a detailed microscopic modeling, which is why random graphs only use partial descriptions of networks, such as degree distributions. Statistical analysis of network data suggests that many networks possess a power-law degree distribution \cite{Clauset2009,Newm10a,Pastor2002,email2002}, where the probability $P(k)$ that a node has $k$ neighbors scales as $ k^{-\tau}$ for some characteristic exponent $\tau>0$. The power-law distribution leads to scale-free behavior such as short distances due to the likely presence of high-degree nodes. Many scale-free networks are reported to have an exponent $\tau$ between 2 and 3 \cite{AlbJeoBar99a,FalFalFal99,Jeoetal00}, so that the second moment of the degree distribution diverges in the infinite-size network limit.

Hidden variable models present a class of popular null models for scale-free networks \cite{park2004statistical,pol2012}.
%he CM was not studied in its original form, but rather in an enlarged ensemble
%referred to as random graphs with hidden variables.
In these models vertices are characterized by hidden variables that influence the creation of edges between pairs of vertices.
The models can be seen as enlarged ensembles of random graphs that can match in expectation any given degree distribution $P(k)$.
All topological properties, including correlations and clustering, then become functions of the distribution of the hidden variables and the probability of connecting vertices \cite{boguna2003,bollobas2007phase}. The independence between edges makes hidden-variable models analytically tractable. One complication though, is that these hidden-variable models come with a rather complicated cutoff scheme for scale-free networks with $\tau<3$. Indeed, to cope with diverging second moments, various cutoff schemes were proposed that remove the problematic large-degree vertices \cite{ChuLu01,catanzaro2005generation}. For clustering this can have a dramatic effect, for instance on degree-degree correlations that are typically of a disassortative nature: high-degree vertices tend to be connected to low-degree vertices \cite{park2004statistical,maslov2002specificity}. This negative correlation can have a strong influence on topological network properties, including clustering, defined as the presence of triangles in the network.

In \cite{pol2012} it was shown that hidden-variable models with a nonrestrictive cutoff scheme can generate nearly size-independent levels of clustering, and can thus generate networks with high levels of clustering, particularly for $\tau$ close to 2. Without banning large-degree vertices by installing a cutoff, the long tail of the power law makes that pairs of high-degree vertices would quite likely share more than one edge (after for instance a random assignment of edges). But hidden-variable models allow at most one edge between pairs of vertices, so that large-degree vertices must inevitably connect to small-degree vertices due to lack of available large-degree vertices. This phenomenon, in turn, is related to the difference in scaling between the so-called structural cutoff and natural cutoff. The structural cutoff is defined as the largest possible upper bound on the degrees required to guarantee single edges, while the natural cutoff characterizes the maximal degree in a sample of $N$ vertices. For scale-free networks with $\tau\in(2,3]$ the structural cutoff scales as $N^{1/2}$ while the natural cutoff scales as $N^{1/(\tau-1)}$ (see Section \ref{setup}), which gives rise to structural negative correlations and possibly other finite-size effects.

%In both studies \cite{park2004statistical,pol2012} the CM was not studied in its original form, but rather in an enlarged ensemble
%referred to as random graphs with hidden variables.

%%In these models vertices are characterized by hidden variables that influence the creation of edges between pairs of vertices.
% Hidden variable models are can be seen as an enlarged ensemble of random graphs that can match in expectation any given degree distribution $P(k)$.  All topological properties, including correlations and clustering, are functions of the distribution of the hidden variables and the probability of connecting vertices \cite{boguna2003,bollobas2007phase}. As it turns out, such ensembles of graphs are more tractable than the CM, because of the independence between edges. One complication though, of these hidden-variable models, is that they come with a rather complicated cutoff scheme for scale-free networks with $\tau<3$. Indeed, to cope with diverging second moments, various cutoff schemes were proposed that remove the problematic large-degree vertices \cite{ChuLu01,catanzaro2005generation}. For clustering this can have a dramatic effect.

Clustering can be measured in various ways. The local clustering coefficient of a vertex is defined as the number of existing
edges among the set of its neighbors divided by the total number of possible connections between them. This can be interpreted as the probability that two randomly chosen neighbors of a vertex are neighbors themselves. One could calculate a global clustering coefficient as the total number of triangles (closed triples of three vertices) divided by the number of connected triples. Here we employ as a metric a different global clustering coefficient, the average clustering coefficient $C$, defined as the average (over vertices of degree $\geq 2$) of the local clustering coefficient of single vertices. For a node $i$ the local clustering coefficient is given by $c_i=2T_i/k_i(k_i-1)$ with $k_i$ the degree of node $i$ and $T_i$ the number of triangles that vertex $i$ is part of.

In the absence of high-degree nodes, the average clustering coefficient is given by \cite{Newm10a}
\begin{equation}\label{Ceq}
C=\frac1N \sum_{i=1}^N c_i=%\frac{(\mathbb{E}D^2-\mathbb{E}D)^2}{N (\mathbb{E}D)^3},
\frac{\mean{k(k-1)}^2}{N\mean{k}^3},
\end{equation}
which shows that clustering vanishes very fast in the large network limit $N\to\infty$ in support of the tree-like approximations of complex networks. However, for scale-free distributions with $\tau<3$, the natural cutoff that scales as $N^{1/(\tau-1)}$ together with (\ref{Ceq}) gives $C\sim N^{(7-3\tau)/(\tau-1)}$. The diverging $C$ for $\tau<7/3$ is caused by the many edges between the high-degree vertices, and can be judged as anomalous or nonphysical behavior if one wants $C$ to be smaller than 1 and interpret it as a probability or proportion.
%Expression (\ref{Ceq}) thus fails because it was derived without accounting for the structural correlations among degrees of connected nodes.
If a structural cutoff of order $N^{1/2}$ is imposed, hence banning the largest-degree nodes, formula (\ref{Ceq}) predicts the correct (in the sense that it matches simulations) scaling $N^{2-\tau}$ \cite{catanzaro2005generation,pol2012}.

In a power-law setting, the infinite variance is essential for describing scale-free network behavior, which makes the banning of large-degree vertices unnatural. In this paper we investigate average clustering for an ensemble of scale-free random graphs that allows for an interplay between structural correlations and large-degree nodes. The clustering coefficient in this ensemble turns out to depend on the size of the network, the structural cutoff that arises when conditioning on simplicity and the natural cutoff that accounts for large degrees.

%
%But for scale-free networks with $\tau<3$ the CM model starts
%
%
%For $\tau<3$, however, the inherent large fluctuations in degrees causes the CM model to display some anomalous behavior with regard to strong nontrivial degree correlations among nodes,... cutoff \cite{Pastor2002,Clauset2009barabasireview,nsw,Newm03a}.
%
%

\section{Hidden variables and cutoffs}\label{setup}

Given $N$ nodes, hidden-variable models are defined as follows: (i) associate to each node a hidden variable $h$ drawn from a given probability distribution function $\rho(h)$ and (ii) join each pair of vertices independently according to a given probability $p(h,h')$ with $h$ and $h'$ the hidden variables associated to the two nodes. The probability $p(h,h')$ can be any function of the hidden variables, as long as $p(h,h')\in[0,1]$. Many networks can be embedded in this hidden-variable framework, but particular attention goes to the case in which the hidden variables have themselves the structure of the degrees of a real-world network. One could interpret the hidden-variable model as yielding soft constraints on the degrees, rather than hard constraints often used in the configuration model \cite{Clauset2009,Newm10a,Pastor2002,email2002,Dhara16}.
Chung and Lu \cite{ChuLu01} introduced this model in the form
\begin{equation}\label{c1}
p(h,h')\sim \frac{h h'}{N \mean{h}},
\end{equation}
so when allowing self-loops and multiple edges, the expected degree of a node equals its hidden variable.
For \eqref{c1} to make sense we need that the maximal value of the product $h h'$ never exceeds $N \mean{h}$ and this can be guaranteed by
the assumption that the hidden degree $h$ is smaller than the structural cutoff $h_s=\sqrt{N\mean{h}}$.
While this restricts $p(h,h')$ within the interval $[0,1]$, the structural cutoff strongly violates the reality of scale-free networks. Regarding the hidden variables as the desired degrees in the CM, the structural cutoff conflicts with the fact that the natural cutoff for the degree scales as $N^{1/(\tau-1)}$.
%Therefore, a form as \eqref{c1} with a structural cutoff $(N\mean{h})^{1/2}$ can never generate an approximate power-law degree distribution $P(k)\sim k^{-\tau}$ for $\tau<3$. \chr{REMCO: depends on how fast $k$ can grow with $N$}

In \cite{boguna2003,park2004statistical,bollobas2007phase} more general hidden-variable models were introduced, constructed  to preserve conditional independence between edges, while making sure there is only one edge between every vertex pair and that the natural extreme values of power-law degrees are not neglected. Within that large spectrum of models, we focus on
%the sparse regime models, in which the number of edges scales linearly in the number of vertices, and in which the phase transition and is a property of many large real-world graphs. We also focus on
the subset of models for which
\begin{equation}\label{rel}
 P(k)= \int_h \frac{\e^{-h} h^k}{k!} \rho(h) \dd h \sim \rho(k).
\end{equation}
 The class of models considered in this paper starts from the ansatz $p(h,h')\approx {h h'}/{N\mean{h}}$, but like \cite{boguna2003,park2004statistical,bollobas2007phase,britton2006generating,NorRei06} \textit{adapts} this setting to incorporate extreme values and to rule out self-loops.

\subsection{Class of random graphs}\label{class}
%We regard our random graphs as indexed by $N$ and consider
%what happens as $N\to\infty$.

Within the wide class of hidden-variable models  \cite{boguna2003,park2004statistical,bollobas2007phase} we consider probabilities of the form
\begin{equation}\label{c2}
p(h,h')=r(u)=u f(u) \quad {\rm with} \quad u=\frac{h h'}{h_s^2}
\end{equation}
with functions $f:[0,\infty)\to (0,1]$ that belong to the F-class spanned by the properties
\begin{itemize}
\item[F1] $f(0)=1$, $f(u)$ decreases to $0$ as $u\to\infty$.
\item[F2] $r(u)=uf(u)$ increases to $1$ as $u\to\infty$.
\item[F3] $f$ is continuous and there are $0=u_0<u_1<\ldots< u_K<\infty$ such that $f$ is twice differentiable on each of the intervals $[u_{k-1},u_k]$ and on $[u_K,\infty)$, where
\begin{equation}
f'(u_k)=\frac 12 f'(u_k+0)+\frac12f'(u_k-0)
\end{equation}
for $k=1,\dots,K$ and $f'(0)=f'(+0)$.

%right and left derivatives are understood at the endpoints of the intervals.
%
%has for every $u\in(0,\infty)$ a finite left and right derivative, $f_L'(u)$ and $f_R'(u)$.
\item[F4] $-uf'(u)/f(u)$ is increasing in $u\geq 0$. %\chr{remark Remco: $f$ is not necessarily cont. diff., see F2}
\end{itemize}
The class of hidden-variable models considered in this paper is completely specified by all functions $f$ that satisfy F1-F4.
Here are important classical members of the F-class:
\begin{itemize}
\item[(i)] ({\it maximally dense graph}) The Chung-Lu setting
\begin{equation}\label{ex1}
r(u)=\min\{u,1\}.
\end{equation}
This is the default choice in \cite{bollobas2007phase} and leads within the F-class to the densest random graphs.
\item[(ii)]({\it Poisson graph}) A simple exponential form gives
\begin{equation}\label{ex2}
r(u)=1-\e^{-u}.
\end{equation}
Here we take $u$ to define the intensities of Poisson processes of edges, and ignore multiple edges, so
that \eqref{ex2} gives the probability that there is an edge between two vertices. Variants of this form are covered in
e.g.~\cite{bollobas2007phase,NorRei06,BhaHofLee09a,BhaHofLee09b}.
\item[(iii)] ({\it maximally random graph}) The next function was considered in \cite{park2004statistical,pol2012,squartini2011analytical}:
\begin{equation}\label{ex3}
r(u)=\frac{u}{1+u}.
\end{equation}
This connection probability ensures that the entropy of the ensemble is maximal \cite{pol2012}. This random graph is also known in the literature as the generalized random graph \cite{britton2006generating,Hofs15}.
\end{itemize}

The conditions F1-F4 will prove to be the minimally required conditions for the results that we present for the clustering coefficient. The F-class is constructed so that it remains amenable to analysis; the technique developed in \cite{pol2012} to characterize the average clustering coefficient despite the presence of correlations can be applied to our class. Notice that the technical condition F3 allows to consider piecewise smooth functions with jumps in their derivatives, such as $\min (1,1/u)$ that comes with \eqref{ex1}.
% Condition F4 can be best understood by observing that
%\begin{equation}
%-uf'(u)/f(u)=1-u r'(u)/r(u)
%\end{equation}
%By F2 we know \chr{remco: How do we know this?} that $u r'(u)/r(u)\to 0$ as $u\to\infty$, so that F4... \chr{DE VERKLARING VAN F4 IS NOG NIET HELDER GENOEG}.
It can be shown that F4 is slightly stronger than the condition of concavity of $r(u)$. It appears in Section~\ref{sec:univ} that F4 is necessary and sufficient for monotonicity in a general sense of the local clustering coefficient $c(h)$.

\subsection{Cutoffs and correlation}\label{cutoff}

The hidden-variable model by definition excludes self-loops and avoids multiple connections after imposing the structural cutoff $h_s\sim N^{1/2}$. Since the natural cutoff is of the order $h_c\sim N^{1/(\tau-1)}$, for $\tau\geq 3$ the structural cutoff dominates and correlations are avoided. For $\tau<3$, however, the structural cutoff is smaller than the natural cutoff predicted by extreme value theory. All actual cutoffs larger than $h_s$ will then result in a network with a structure that can only be analyzed by considering non-trivial degree-degree correlations.

The structural cutoff $h_s$ now marks the point as of which correlations imposed by the network structure arise. All pairs of vertices with hidden variables smaller than this cutoff are connected with probability close to $u=h h'/h_s^2$ and do not show degree-degree correlation. The extent to which the network now shows correlation is determined by the gap between the natural cutoff $h_c$ and the structural cutoff $h_s$. A fully uncorrelated network arises when $h_c<h_s$, while correlation will be present when $h_c>h_s$.
Let $\mean{h}$ denote the average value of the random variable $h$ with density $\rho(h)=Ch^{-\tau}$ on $[h_{\min},N]$, so that
\begin{equation}\label{eq:meanh}
\mean{h}=\frac{\int_{h_{\min}}^{N}h^{1-\tau}\dd h}{\int_{h_{\min}}^{N}h^{-\tau}\dd h} = \frac{\tau-1}{\tau-2}\frac{h_{\min}^{2-\tau}-N^{2-\tau}}{h_{\min}^{1-\tau}-N^{1-\tau}}.
\end{equation}
With the default choices
\begin{align}
h_s&=\sqrt{N\mean{h}},\label{def2}\\
h_c&=(N\mean{h})^{1/(\tau-1)},\label{def3}
\end{align}
in mind, the regime in terms of cutoffs we are interested in is, just as in \cite{pol2012},
\begin{equation}\label{desiredregime}
h_s\leq h_c\ll h_s^2,
\end{equation}
where we regard these cutoffs as indexed by $N$ and consider what happens as $N\to\infty$, with emphasis on the asymptotic regime
$h_s\ll h_c$ for $N$ large.

In Appendix~\ref{sec:natcut} we show that $h_c$ as given in~\eqref{def3} is an accurate approximation of $\mathbb{E}[\max(\underline{h}_1,\dots,\underline{h}_N)]$, where the $\underline{h}_i$ are i.i.d. with $\rho(h)$ as density.

\section{Universal properties}\label{sec:univ}
For the class of hidden-variable models described in Section \ref{class}, we will characterize the large-network asymptotics of the local clustering coefficient $c(h)$ and average clustering coefficient $C$. The first result in this direction for $C$ was obtained for a class of uncorrelated random scale-free networks with a cutoff of $N^{1/2}$ \cite{catanzaro2005generation} for which $C$ turned out to scale as $N^{2-\tau}$, a decreasing function of the network size for $\tau>2$. In \cite{pol2012} the more general setup discussed in Section \ref{setup} was used, with the specific choice of $r(u)=u/(1+u)$. After involved calculations with Lerch's transcendent, \cite{pol2012} revealed the scaling relation
\begin{equation}\label{polsim}
C\sim h_s^{-2(\tau-2)}\ln(h_c/h_s).
\end{equation}
For the default choices \eqref{def2} and \eqref{def3} this predicts $C\sim N^{2-\tau}\ln N$ (ignoring the constant).
%\chr{WE HAVE THE TERM $\ln(h_c^2/h_s^2)$; ONLY CHANGES THE CONSTANT}

We adopt the hidden variables formalism developed in \cite{boguna2003} that leads, among other things, to explicit expressions for
the local clustering coefficient $c(\w)$ of a node with hidden variable $\w$ and for
the average clustering coefficient $C$.

The clustering coefficient of a vertex with hidden variable $h$ can be interpreted as the probability that two randomly chosen edges from $h$ are neighbors. The clustering of a vertex of degree one or zero is defined as zero. Then, if vertex $h$ has degree at least two,
\begin{equation}\label{}
c(h)=\int_{h_{\min}}^{h_c}\int_{h_{\min}}^{h_c} p(h'|h)p(h',h'')p(h''|h)\dd h'\dd h''
\end{equation}
with $p(h'|h)$ the conditional probability that a randomly chosen edge from an $h$ vertex is connected to an $h'$ vertex given by
\begin{equation}\label{}
p(h'|h)=\frac{\rho(h')p(h,h')}{\int_{h''} \rho(h'')p(h,h'')\dd h''}.
\end{equation}
Furthermore, the probability that a vertex with hidden variable $h$ has degree at least two is given by
	\begin{equation}
	\mathbb{P}(k\geq 2\mid h)=\sum_{k=2}^\infty \frac{h^k\e^{-k}}{k!}=1-\e^{-h}-h\e^{-h}.
	\end{equation}
Therefore, for $\rho(h)\sim h^{-\tau}$ \cite[Eq.~(29)]{boguna2003}
\begin{equation}\label{}
\begin{aligned}
&c(\w)=(1-\e^{-h}-h\e^{-h}) \frac{\int_{{h_{\rm min}}}^{h_c}\int_{{h_{\rm min}}}^{h_c}\rho(\w')p(\w,\w')\rho(\w'')p(\w,\w'')p(\w',\w'')\dd \w'\dd \w''}{\Big[\int_{{h_{\rm min}}}^{h_c}\rho(\w')p(\w,\w')\dd \w'\Big]^2},
\end{aligned}
\end{equation}
%For the probability $P_N(k)$ that a vertex has degree $k$ in a network with $N$ nodes we know that \cite{boguna2003}
%\begin{equation}\label{}
%P_N(k)\to \int_h \frac{\e^{-h} h^k}{k!} \rho(h)\dd h=P(k) \quad {\rm as} \ N\to\infty.
%\end{equation}
%The clustering coefficient of a vertex with degree $k$ is then given by
%\begin{equation}\label{}
%\bar c(k)=\frac{1}{P(k)}\int_h \frac{\e^{-h} h^k}{k!} \rho(h) c(h) \dd h,
%\end{equation}
and hence
\begin{align}\label{eq:intC}
%C=\sum_k \bar c(k)P(k)=\int_h \rho(h) c(h) \dd h,
C=\int_{h_{\min}}^{h_c} \rho(h) c(h) \dd h.
\end{align}
The degree of a vertex conditioned on its hidden variable $h$ is distributed as a Poisson random variable with parameter $h$~\cite{boguna2003} and~\cite[Chapter 6]{Hofs15}.
Note that the Poisson distribution is sharply peaked around $k=h$, which for large $k$ yields
\begin{align}\label{}
P(k)\sim \rho(k) \quad {\rm and} \quad \bar c(k) \sim c(k),
\end{align}
where $\bar{c}(k)$ denotes the average clustering coefficient over all vertices of degree $k$, so that the hidden variables become hidden degrees.

We make the change of variables
\begin{equation}\label{eq:change}
a=\frac{1}{h_s}, \quad b=\frac{h_c}{h_s}
\end{equation}
and assume henceforth, in line with \eqref{desiredregime}, that
\begin{equation}\label{assonab}
0<a h_{\min}\leq a h_{\min}b \leq 1\leq b<\infty, \quad 2<\tau<3.
\end{equation}
This gives $c(h)=(1-\e^{-h}-h\e^{-h})c_{ab}(h)$ with
\begin{equation}\label{mainc}
c_{ab}(\w)=\frac{\int_{a {h_{\rm min}}}^{b}\int_{a {h_{\rm min}}}^{b} (xy)^{-\tau}r(a \w x)r(a\w y)r(xy)\dd x\dd y}{\Big[\int_{a {h_{\rm min}} }^{b}x^{-\tau}r(a \w x)\dd x\Big]^2}.
\end{equation}
Note that within the domain of integration $[ah_{\min},b]$ in Eq.~\eqref{mainc} the arguments $a \w x$ and $a \w y$ do not exceed a maximum value $O(ab)$ as long as $\w<h_s^2/h_c$, which tends to zero under assumption \eqref{desiredregime}.
Therefore, since $r(u)\approx u$,  $c_{ab}(\w)\approx c_{ab}(0)$ for $\w < h_s^2/h_c$. When choosing $h_s$ as in~\eqref{def2}, this means that $c_{ab}(\w)\approx c_{ab}(0)$ for $\w\leq  N\mean{h}/h_c$. 
%In particular, when there is no natural cut-off, $h_c=\infty$, the constant regime for $c_{ab}(\w)$ disappears.
In Proposition~\ref{thm2} below we will prove that $h\mapsto c_{ab}(\w)$ is a bounded monotonically decreasing function for the class of models at hand. 
Furthermore, the density $\rho(h)\sim h^{-\tau}$ with $\tau\geq 2$, decays sufficiently rapidly for the integral in~\eqref{eq:intC} for $C$ to have converged already before $c_{ab}(h)$ starts to drop significantly below its value $c_{ab}(0)$ at $h=0$. Thus, $C$ can be approximated with
%\begin{equation}
%C(\tau)=c(0)=\frac{\int_{a_N}^{b_N}\int_{a_N}^{b_N} (xy)^{2-\tau}\frac{1}{}\dd x\dd y}{\Big[\int_{a_N}^{b_N}x^{-\tau}r(a \w x)\dd x\Big]^2}
%\end{equation}
\begin{equation}\label{e2}
\begin{aligned}[b]
C_{ab}(\tau)&=c_{ab}(0)\int_{h_{\min}}^{N}\rho(\w)(1-(1+\w)\e^{-\w})\dd\w:=c_{ab}(0)A(\tau),
\end{aligned}
\end{equation}
where we have conveniently extended the integration range to the $\tau$-independent interval $[h_{\min},N]$ at the expense of a negligible additional error.

\begin{proposition}\label{thm2}
Assume that $f$ satisfies {\rm F1 - F3}. Then $c_{ab}(h)$ is decreasing in $\w\geq 0$ for all $a,b$ with $0<a< b$ if and only if $f$ satisfies {\rm F4}.
\end{proposition}

Proposition \ref{thm2} shows that for large enough $h$ local clustering decreases with the hidden variable. For the default choices  with \eqref{ex1}, \eqref{def2}, \eqref{def3}, local clustering $c(\w)$ is plotted in Figure \ref{fig1}, which shows both exact formulas and extensive simulations. Because our model starts from a single-edge constraint, Proposition \ref{thm2} also provides support for the asserted dissassortative mixing observed in many technological and biological networks \cite{park2004statistical,maslov2002specificity}.
%Hence, although interesting in its own right, in this paper Proposition \ref{thm2} serves as an important prerequisite for the next result.

%While the multi-edge CM does not have this, the CM model is this paper is able to describe... has also consequences for targeted attacks and disease spreading...

\begin{proposition}\label{thm1}
Assume that $f$ is positive, and satisfies {\rm F3}. Then $c_{ab}(0)$ is decreasing in $\tau>0$ for all $a,b$ with $0<a\leq b<\infty$ if and only if $f$ satisfies {\rm F2}.
\end{proposition}

%\begin{remark}
%\chr{kijken of we hier iets mee willen doen} There is the exact representation
%\begin{align}
%&\int_{a}^{b}\int_{a}^{b} (xy)^{2-\tau}f(xy)\dd x\dd y=\int_{a^2}^{ab}z^{2-\tau}f(z)\ln (z/a^2)\dd z\nonumber\\
%&+\int_{ab}^{b^2}z^{2-\tau}f(z)\ln (b^2/z)\dd z.
%\end{align}
%\end{remark}

Proposition~\ref{thm1} gives evidence for the fact that clustering increases as
$\tau$ decreases, as confirmed in Figure~\ref{fig2}.
More precisely, Proposition~\ref{thm1} shows the monotonicity of $c_{ab}(0)$, which is one of the factors of $C_{ab}(\tau)$ in~\eqref{e2}.
The issue of monotonicity of $C_{ab}(\tau)$ is more delicate, since $a$ and $b$ are function of $\tau$ themselves.
In Appendix~\ref{sec:monotone} we present several other monotonicity properties of the remaining building blocks that together give $C_{ab}(\tau)$. It follows that for $\tau>2$, $C_{ab}(\tau)$ is bounded from above by an envelope function of $\tau$ that is very close to $C_{ab}(\tau)$  and that is decreasing in $\tau$.
%The next theorem builds on Proposition~\ref{thm1}, but $a$ and $b$ should be interpreted, in line with \eqref{def2}-\eqref{def3} as $a(\tau) = 1/h_s(\tau)$ and $b(\tau) = h_c(\tau)/h_s(\tau)$:
%
%\begin{theorem}\label{thm3}
%Assume that $f$ is positive, and satisfies {\rm F3}. Then $C_{ab}(\tau)$ is decreasing in $\tau>0$ if and only if $f$ satisfies {\rm F2}.
%\end{theorem}
%Proposition 3 proves monotonicity of average clustering for a large class of hidden-variable models.
\begin{figure}[t]
	\begin{minipage}[t]{0.45\linewidth}
		\centering
		\includegraphics[width=\textwidth]{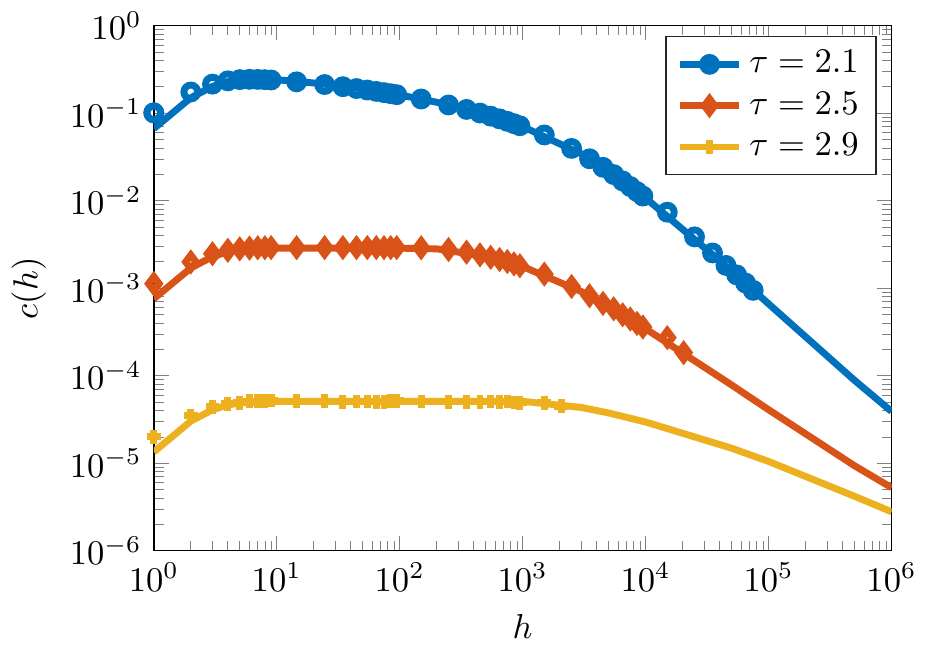}
		\caption{$c(\w)$ for $\tau=2.1,2.5,2.9$ and networks of size $N=10^6$, using $h_{\min}=1$. The markers indicate the average of $10^5$ simulations, and the solid lines follow from~\eqref{e2}.}
		\label{fig1}
	\end{minipage}
	\hfill
	\begin{minipage}[t]{0.45\linewidth}
		\centering
		\includegraphics[width=\textwidth]{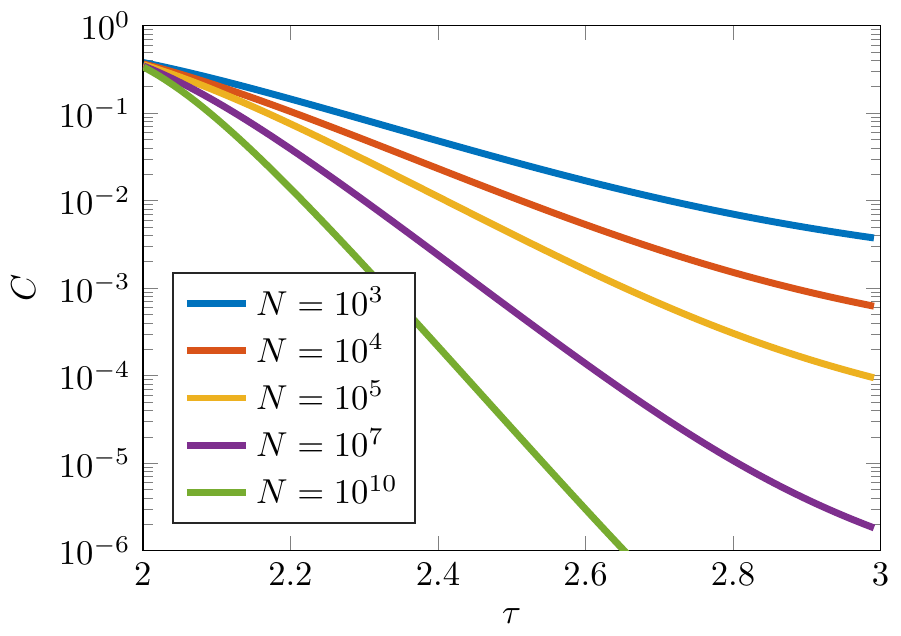}
		\caption{$C_{ab}^{\max}(\tau)$ as a function of $\tau$, with $r$ as in \eqref{ex1}, using $h_{\min}=1$.}
		\label{fig:Ctau}
	\end{minipage}
\end{figure}
Figure~\ref{fig:Ctau} provides empirical evidence for the monotonicity of $C_{ab}(\tau)$ in $\tau$.  This monotonicity seems to conflict observations made in~\cite{pol2012}, where the clustering coefficient of a hidden-variable model first increases in $\tau$ when $\tau$ is close to 2, and then starts decreasing. The difference is caused by the choice of the structural cutoff. Where we take $h_s=\sqrt{N\mean{h}}$ with $\mean{h}$ as in~\eqref{eq:meanh}, in~\cite{pol2012} $h_s=\sqrt{N(\tau-1)/(\tau-2)}$ was used. Thus, in~\cite{pol2012}, the structural cutoff includes the infinite system size limit of $\mean{h}$, where we use the size-dependent version of $\mean{h}$. 
%This small difference in the choice of the structural cutoff causes the different behavior near $\tau=2$.

Figure~\ref{fig2} suggests that $C$ falls off with $N$ according to a function $N^{\delta}$ where $\delta$ depends on $\tau$.
In Proposition~\ref{propCmax} below, we will show that for the F-class of hidden-variable models and the standard cutoff levels, $C$ decays as $N^{\tau-2}\ln N$. On a log-scale, moreover, the clustering coefficient of different hidden-variable models in the F-class only differs by a constant, which is confirmed in Figure~\ref{fig2} and substantiated in Proposition~\ref{propCmaxbound}.
Then, we will focus in Section~\ref{sec:pers} on $\tau\approx 2$, for which Figure~\ref{fig2} suggests that the clustering remains nearly constant as a function of $N$, and characterize how large a network should be for $C$ to start showing decay. This again will depend on $\tau$.

\begin{figure}[t]
	\centering
	\includegraphics[width=0.45\textwidth]{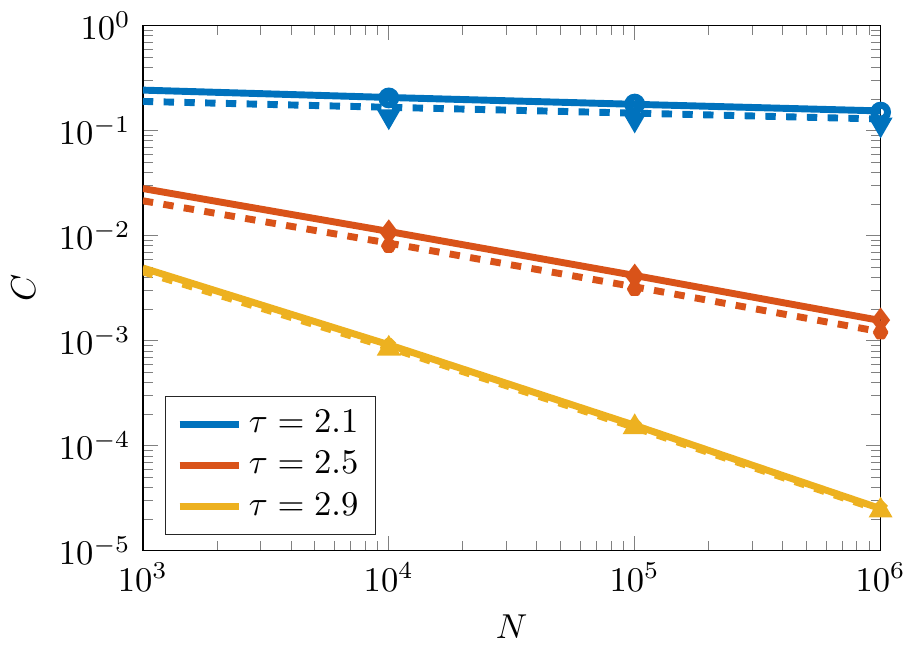}
	\caption{$C_{ab}(\tau)$ for $\tau=2.1,2.5,2.9$, choices \eqref{ex1} (solid line) and \eqref{ex3} (dashed line) and networks of size $N=10^k$ for $k=4,5,6$, using $h_{\min}=1$. The markers indicate the average of $10^5$ simulations, and the solid lines follow are~\eqref{frontt} and~\eqref{newpol}.}
	\label{fig2}
\end{figure}

\section{Universal bounds}\label{sec:max}
We next compute the clustering coefficient $C_{ab}^{\rm max}(\tau)=C_{ab}(\tau)$ for the maximally dense graph with $f(u)=f_{\rm max}(u)=\min(1,1/u)$, $u\geq 0$ and $a,b$ satisfying \eqref{assonab}.
In this case we have $c_{ab}(h)=c_{ab}(0)$ for $h\leq 1/(ab)= h_s^2/h_c=N\mean{h}/h_c$. It is easy to see that $f_{\rm max}$
is the maximal element in the F-class in the sense that $f(u)\leq f_{\rm max}(u)$ for all $u\geq 0$ and all $f\in{\rm F}$. For a general $f\in{\rm F}$  we shall also bound $C_{ab}^f(\tau)$ in terms of $C_{ab}^{\rm max}(\tau)$. This yields a scaling relation similar to \eqref{polsim}, but then for the whole F-class. We start from an explicit representation for $C_{ab}^{\rm max}(\tau)$:

\begin{proposition}\label{propCmax}
\begin{align}\label{frontt}
C_{ab}^{\rm max}(\tau)=\frac{A(\tau)(\tau-2)^2}{\left((ah_{\min})^{2-\tau}-b^{2-\tau}\right)^2}
\times I_{ab}^{\rm max}(\tau),
\end{align}
with $A(\tau)$ given in~\eqref{e2} and
\begin{align}\label{c12}
I_{ab}^{\rm max}(\tau)&=
\frac{\ln (b^2)}{(\tau-2)(3-\tau)}
-\frac{1-b^{2(2-\tau)}}{(\tau-2)^2}+\frac{1-2(ah_{\min}b)^{3-\tau}+(ah_{\min})^{2(3-\tau)}}{(3-\tau)^2}.
\end{align}
\end{proposition}

In Appendix \ref{sec:pol} we give the counterpart of \eqref{frontt} for the maximally random graph \eqref{ex3} studied in \cite{pol2012}, and show that on a log-scale the leading asymptotics differs only by a constant, so that the decay exponent describing how the clustering decays with network size is the same. This can also be seen in Figure~\ref{fig2}. In fact, for all functions $f$ in the F-class we show below that the decay exponent is universal, and that the difference in constants can be bounded.

When $\tau$ is away from $2$ and $3$, and $b$ is large and $a$ is small, we can ignore the $b^{2-\tau}$ in the front factor of \eqref{frontt} and the second term in \eqref{c12}. Furthermore, $ab=O(N^{\frac{2-\tau}{\tau-1}})$, so that we may ignore this factor in the third term of~\eqref{c12}. In this case we get the approximation
\begin{align}\label{mainapprox}
C_{ab}^{\rm max}(\tau)\approx A(\tau)\frac{\tau-2}{3-\tau}(ah_{\min})^{2(\tau-2)}\ln (b^2).
\end{align}
Using the default choices for $a$ and $b$ from~\eqref{def2},~\eqref{def3} and~\eqref{eq:change} then shows that $C_{ab}^{\max}\sim N^{2-\tau}\ln(N)$ (ignoring the constant).
Proposition~\ref{propCmax}
can be used to find upper and lower bounds for $C_{ab}^f(\tau)$ with general $f\in{\rm F}$. Since $f(u)\leq f_{\rm max}(u)$, $u\geq 0$, it follows from~\eqref{e2} that
\begin{align}\label{c14}
C_{ab}^f(\tau)\leq C_{ab}^{\rm max}(\tau).
\end{align}

%\chr{Remco: Discuss that for $C$ this is not obvious, because of the ratio in $c_{ab}(h)$}

\begin{proposition}\label{propCmaxbound}
For all $u_0\geq 1$,
\begin{align}\label{frontt2}
C_{ab}^f(\tau)\geq u_0 f(u_0)C_{a_0b_0}^{\rm max}(\tau),
\end{align}
with $a_0=a/\sqrt{u_0}$ and $b_0=b/\sqrt{u_0}$.
\end{proposition}
In particular, the choice $u_0=1$ yields
\begin{align}\label{c20}
C_{ab}^f(\tau)&\geq f(1)C_{ab}^{\rm max}(\tau),
\end{align}
which together with \eqref{mainapprox} and \eqref{c14} gives the large-network behavior of $C_{ab}^f(\tau)$, when $\tau\in(2,3)$ and away from 2 and 3, and $a$ is small and $b$ is large (up to a multiplicative constant that is less interesting).
In particular, this shows that $C_{ab}^f\sim N^{2-\tau}\ln(N)$ (again ignoring the constant).

%A similar computation as the computation that leads to~\eqref{frontt} and~\eqref{mainapprox} can approximate the clustering coefficient when $b=\infty$, the pure power-law case. This shows that the scaling $C\sim N^{2-\tau}\ln N$ still holds when there is no natural cut-off.

\section{Persistent clustering}\label{sec:pers}
In \cite{pol2012} it was observed that for values of the exponent $\tau\approx 2$, clustering remains nearly constant up to extremely large network sizes, which makes the convergence to the hydrodynamic limit extremely slow (as also observed in \cite{boguna2009langevin,janssen2016giant}). Here we now use the explicit results for the maximally dense graph to characterize this rate of convergence as a function of the network size $N$. For convenience, we assume in this section that $h_{\min}=1$.

In view of the lower and upper bounds obtained in Section~\ref{sec:max} for $C_{ab}^f(\tau)$  with general $f\in{\rm F}$ in terms of $C_{ab}^{\rm max}(\tau)=C_{ab}(\tau)$, it suffices to consider $C_{ab}^{\rm max}(\tau)$ for $\tau$ close to $2$.
In Appendix \ref{sec:leading} we show that when $\tau$ is close to 2 and $\abs{\ln(ab)/\ln(b^2)}$ is small, $C_{ab}^{\max}(\tau)$ can be approximated by
\begin{align}\label{f9}
C_{ab}^{\rm max}(\tau)\approx \frac{ A(\tau)(1-\frac13 (\tau-2)\ln (b^2))}{2(1-\frac 12(\tau-2)\ln(ab)+\frac16 (\tau-2)^2\ln^2 b)^2}.
\end{align}
It is the term $-\tfrac13 (\tau-2)\ln (b^2)$ in the numerator and the term $\tfrac16 (\tau-2)^2\ln^2 b$ in the denominator (the term $\tfrac12 (\tau-2)\ln(ab)$ being less important) of the right-hand side of \eqref{f9} that are the main influencers for when $C_{ab}^{\rm max}(\tau)$ starts to decay. The decay is certainly absent as long as the numerator $1-\tfrac13(\tau-2)\ln (b^2)$ is away from zero.

We then apply this reasoning to the canonical choices $h_s=\sqrt{N\mean{h}}$ and $h_c=(N\mean{h})^{1/(\tau-1)}$, for which
\begin{align}\label{f12}
b=(N\mean{h})^{\frac{3-\tau}{2(\tau-1)}},\quad ab=(N\mean{h})^{-\frac{\tau-2}{\tau-1}},
\end{align}
ensuring $\abs{\ln(ab)/\ln(b^2)}=(\tau-2)/(\tau-3)$ to be small indeed.  
 Then, choosing a threshold $t\in(0,3)$ and solving $N$ from
 \begin{equation}
 (\tau-2)\ln(b^2)=t,
 \end{equation}
 we get
 \begin{equation}\label{a2}
 N\mean{h}=\exp\left(\frac{\tau-1}{(\tau-2)(3-\tau)}t\right).
 \end{equation}
In Table~\ref{tab1} we consider the case that $t=2$ and use that $\mean{h}$ can accurately be bounded above by $\ln N$ when $h_{\min}=1$ and $\tau$ is close to 2, and we let $N_{\tau,2}$ be such that $N\ln N$ equals the right-hand side of~\eqref{a2}. For $\tau=2.1$, the value of $N$ where the clustering starts to decay is much larger than the typical size of real-world network data sets. This supports the observation that clustering is persistent for $\tau$ close to 2. 

%\begin{table}[h]
%\centering
%\begin{tabular}{r|cccc}
%\hline
%$\tau$ &  $2.3$ & $2.2$ & $2.1$ & $2.05$  \\
%$N_c$ &  $3.63\cdot 10^6$ & $4.81 \cdot 10^7$ & $7.74\cdot 10^{11}$ & $4.59\cdot 10^{20}$ \\
%\hline
%\end{tabular}%
%\caption{Solution $N_{\tau,t}$ to $(\tau-2)\ln (b^2)=t=2$.}
%\label{tab1}
%\end{table}

\begin{table}[t]
\centering
\begin{tabular}{r|l}
\hline
$\tau$ & $N_{\tau,2}$ \\ \hline
   $2.3$    &  $2.37\cdot 10^4$ \\
   $2.2$    & $2.62 \cdot 10^5$ \\
   $2.1$    & $1.93\cdot 10^{9}$ \\
   $2.05$    & $3.92\cdot 10^{17}$ \\
\hline
\end{tabular}%
\caption{Solution $N_{\tau,t}$ to $(\tau-2)\ln (b^2)=2$.}
\label{tab1}
\end{table}

\section{Proof of Proposition \ref{thm1}}\label{sec:proof}
%From \eqref{e2} we obtain for $C_{ab}'(\tau)=\frac{\dd}{\dd\tau}C_{ab}(\tau)$
We consider
\begin{equation}
c_{ab}(0)=D_{ab}(\tau)=\frac{\int_{a}^{b}\int_{a}^{b}f(xy)(xy)^{2-\tau}\dd x\dd y}{\left(\int_{a}^{b}x^{1-\tau}\dd x\right)^2},
\end{equation}
where we have written the lower integration limit $ah_{\min}$ in~\eqref{frontt} as $a$ for notational convenience. We fix $N$, and study the dependence of $D_{ab}(\tau)$ on $\tau$. We assume here that $a$ and $b$ are fixed, and do not depend on $\tau$. Assume that $f$ is positive and satisfies F2 and F3. We have for $D'_{ab}(\tau)=\frac{\dd}{\dd \tau}D_{ab}(\tau)$,
\begin{align}\label{e3}
&D_{ab}'(\tau)=\frac{-\int_{a}^{b}\int_{a}^{b} f(xy)\ln(xy)(xy)^{2-\tau}\dd x\dd y}{\Big(\int_{a}^{b}x^{1-\tau}\dd x \Big)^2}+\frac{2\int_{a}^{b}\int_{a}^{b} f(xy)(xy)^{2-\tau}\dd x\dd y \int_{a}^{b} x^{1-\tau}\ \ln x\dd x}{\Big(\int_{a}^{b}x^{1-\tau}\dd x \Big)^3}
\end{align}
Observe that $D_{ab}'(\tau)\leq 0$ if and only if
\begin{align}\label{e4}
\frac{\int_{a}^{b}\int_{a}^{b} f(xy)\ln(xy)(xy)^{2-\tau}\dd x\dd y}{\int_{a}^{b}\int_{a}^{b} f(xy)(xy)^{2-\tau}\dd x\dd y}
\geq 2 \frac{\int_{a}^{b} x^{1-\tau}\ln x \ \dd x}{\int_{a}^{b} x^{1-\tau}\dd x}.
\end{align}
Symmetry of $f(xy)/(xy)^{\tau-2}$ and $\ln(xy)=\ln x+\ln y$ gives
\begin{align}\label{e5}
&\int_{a}^{b}\int_{a}^{b} f(xy)\ln(xy)(xy)^{2-\tau}\dd x\dd y =2 \int_{a}^{b}\ln x \Big(\int_{a}^{b} f(xy)(xy)^{2-\tau}\dd y \Big)\dd x.
\end{align}
Letting
\begin{alignat}{3}\label{e6}
&W(x)&&=\frac{\int_{a}^{b}f(xy)(xy)^{2-\tau}\dd y}
{\int_{a}^{b}\int_{a}^{b} f(vy)(vy)^{2-\tau}\dd v\dd y},  \ && a\leq x \leq b,\\
\label{e7}
&V(x)&&=\frac{x^{1-\tau}}
{\int_{a}^{b}v^{1-\tau}\dd v},  \ && a\leq x \leq b,
\end{alignat}
we thus need to show that
\begin{align}\label{e8}
\int_{a}^{b}\ln x \  W(x)\dd x\geq \int_{a}^{b}\ln x  \ V(x)\dd x.
\end{align}
Observe that
\begin{align}\label{e9}
x^{\tau-1}\int_{a}^{b}f(xy)(xy)^{2-\tau}\dd y =
\int_{a}^{b}\left(xy f(xy)\right)y^{1-\tau}\dd y,
\end{align}
which increases in $x>0$ when $f$ satisfies F2. Therefore, $W(x)/V(x)$ increases in $x>0$ when $f$ satisfies F2. Furthermore, $\ln x$ increases in $x>0$, so the inequality in \eqref{e8} follows from the following lemma:

\begin{lemma}\label{lemma1}
Let $0<a<b$ and assume that $p(x)$ and $q(x)$ are two positive, continuous probability distribution functions {\rm(}pdf's{\rm )} on $[a,b]$ such that $p(x)/q(x)$ is increasing in $x\in[a,b]$. Let $g(x)$ be an increasing function of $x\in[a,b]$. Then
\begin{align}\label{e10}
g_p=\int_a^b g(x)p(x)\dd x \geq \int_a^b g(x)q(x)\dd x = g_q.
\end{align}
\end{lemma}
\begin{proof}
For any $R\in \mathbb{R}$,
\begin{align}\label{e11}
g_p-g_q=\int_a^b (g(x)-g_q)(p(x)-R q(x))\dd x,
\end{align}
since $p$ and $q$ are pdf's. Let $x_q$ be a point in $[a,b]$ such that $g(x)\leq g_q$ when $x\leq x_q$ and $g(x)\geq g_q$ when $x\geq x_q$. Choose $R=p(x_q)/q(x_q)$, so that by monotonicity of $g$ and $p/q$, there are the logical representations
\begin{align}
a\leq x\leq x_q &\Rightarrow \left(g(x)-g(x_q)\leq 0 \wedge p(x)-R q(x)\leq 0\right),\nonumber\\
\label{e13}
x_q\leq x\leq b &\Rightarrow \left(g(x)-g(x_q)\geq 0 \wedge p(x)-R q(x)\geq 0\right).\nonumber
\end{align}
Hence, the integrand in \eqref{e11} is everywhere nonnegative, so that $g_p-g_q\geq 0$ as required.
\end{proof}

\begin{remark} The following observation will prove useful later:
{\rm(i)} The inequality in \eqref{e10} is strict when both $g(x)$ and $p(x)/q(x)$ are strictly increasing. {\rm(ii)} When $g(x)$ is (strictly) decreasing and $p(x)/q(x)$ is (strictly) increasing, there is $\leq$ ($<$) rather than $\geq$ ($>$) in \eqref{e10}.
\end{remark}

Now that we have shown F2 to be a sufficient condition for $D_{ab}(\tau)$ to be increasing, we next show that F2 is also a necessary condition. Suppose we have two points $u_1$ and $u_2$ with $0<u_1<u_2$ such that $u_1f(u_1)>u_2f(u_2)$. Since $f$ is continuous and piecewise continuous differentiable, there is a $u_0\in(u_1,u_2)$ and $\varepsilon>0$ such that $uf(u)$ is strictly decreasing in $u\in[u_0-\varepsilon,u_0+\varepsilon]$. In \eqref{e9}, take $a=\sqrt{u_0-\varepsilon}$ and  $b=\sqrt{u_0+\varepsilon}$ so that $xy\in [u_0-\varepsilon,u_0+\varepsilon]$ when $xy\in[a,b]$. Therefore, $W(x)/V(x)$ is strictly decreasing in $x\in[a,b]$. By the version of Lemma \ref{lemma1} with $g(x)=\ln x$ strictly increasing and $p(x)/q(x)=W(x)/V(x)$ strictly decreasing, we see that $g_p-g_q<0$. Therefore, we have \eqref{e8} with $<$ instead of $\geq$, and so $D_{ab}'(\tau)$ is positive for all $\tau$ with this particular choice of $a$ and $b$. This completes the proof of Proposition \ref{thm1}(i).$\hfill\qedsymbol$

\section{Proof of Proposition~\ref{thm2}}
We consider for a fixed $a,b,\tau$ and $\w>0$,
\begin{equation}\label{e14}
c_{ab}(\w)=\frac{\int_{a}^{b}\int_{a}^{b} (xy)^{2-\tau}
f(a \w x)f(a\w y)f(xy)\dd x\dd y}{\Big(\int_{a}^{b}x^{1-\tau}f(a \w x)\dd x\Big)^2}.
\end{equation}
Observe that $\frac{\dd}{\dd h}c_{ab}(h)\leq 0$ if and only if
\begin{align}\label{e15}
&\frac{\frac{\dd}{\dd h}\Big[\int_{a}^{b}\int_{a}^{b} (xy)^{2-\tau}
f(a \w x)f(a\w y)f(xy)\dd x\dd y\Big]}{\int_{a}^{b}\int_{a}^{b} (xy)^{2-\tau}
f(a \w x)f(a\w y)f(xy)\dd x\dd y}\quad \leq {2} \frac{\frac{\dd}{\dd h}\Big[\int_{a}^b f(a h x)x^{1-\tau}\dd x\Big]}{\int_{a}^b f(a h x)x^{1-\tau}\dd x}.
\end{align}
Using
\begin{align}\label{e16}
\frac{\dd}{\dd h}[f(a h x)f(a h y)]&=axf'(ahx)f(ahy)+ay f(ahx)f'(ahy)
\end{align}
and the symmetry of the function $f(xy)/(xy)^{\tau-2}$ gives
\begin{align}\label{e17}
&\frac{\dd}{\dd h}\Big[\int_{a}^{b}\int_{a}^{b}
f(a \w x)f(a\w y)f(xy)(xy)^{2-\tau}\dd x\dd y\Big]\nonumber\\
&\quad =2 \int_{a}^{b}\int_{a}^{b}
a x f'(a \w x)f(a\w y)f(xy)(xy)^{2-\tau}\dd x\dd y.
\end{align}
Also,
\begin{align}\label{e18}
\frac{\dd}{\dd h}\Big[\int_{a}^{b}f(ahx)x^{1-\tau}\dd x ]=\int_{a}^{b}
a x f'(a hx)x^{1-\tau}\dd x.
\end{align}
So we write the left-hand side of \eqref{e15} as
\begin{align}\label{e19}
2\int_{a}^{b}\frac{ax f'(ahx)}{f(ahx)}T(x)\dd x,
\end{align}
and the right-hand side of
\eqref{e15} as
\begin{align}\label{e20}
2\int_{a}^{b}\frac{ax f'(ahx)}{f(ahx)}U(x)\dd x,
\end{align}
where the pdf's $T(x)$ and $U(x)$ on $[a,b]$ are defined as
\begin{align}\label{e21}
T(x)=
\frac{ f(a \w x)\int_a^b f(a\w y)f(xy)(xy)^{2-\tau}\dd y}
{\int_{a}^{b}\int_{a}^{b}
f(a \w v)f(a\w y)f(vy)(vy)^{2-\tau}\dd v\dd y}
\end{align}
and
\begin{align}\label{e22}
U(x)=
\frac{ f(a \w x) x^{1-\tau}}
{\int_{a}^{b}f(a \w v)v^{1-\tau}\dd v}.
\end{align}
The inequality in \eqref{e15} thus becomes
\begin{align}\label{e23}
\int_{a}^{b}\frac{-a h x f'(ahx)}{f(ahx)}T(x)\dd x
\geq \int_{a}^{b}\frac{-a h x f'(ahx)}{f(ahx)}U(x)\dd x,
\end{align}
where we have multiplied by $h>0$. Assume that $f$ satisfies F2. Then
\begin{align}\label{e24}
&x^{\tau-1}\int_{a}^{b}f(ahy)f(xy)(xy)^{2-\tau}\dd y =
\int_{a}^{b}xyf(ahy)f(xy)y^{1-\tau}\dd y
\end{align}
is increasing in $x>0$. Therefore, see \eqref{e21} and \eqref{e22}, $T(x)/U(x)$ is increasing in $x>0$. Hence, from Lemma \ref{lemma1} we get \eqref{e23} when $g(x)=-ahx f'(ahx)/f(ahx)$ is increasing in $x>0$, i.e.~when $f$ satisfies F4.

We have now shown that when $f$ satisfies F1-F3, the condition F4 is sufficient for $c_{ab}(h)$ to be decreasing in $h>0$.
For the result in the converse direction we argue as follows. The function $uf(u)$ is continuous, piecewise smooth, increasing and not constant, and so there is a $u_0>0$, $\varepsilon>0$ such that $uf(u)$ is strictly increasing in $u\in[u_0-\varepsilon,u_0+\varepsilon]$.
Let $z(v)=-vf'(v)/f(v)$, and assume there are $0<v_1<v_2$ such that
$z(v_1)>z(v_2)$. We may assume that $z$ is continuous at $v=v_1,v_2$. Indeed, when $z$ is discontinuous at $v_1$ say, $z(v_1)=\frac12 (z(v_1+0)+z(v_2-0))$ and so at least one of $z(v_1-0)=\lim_{v\uparrow v_1}z(v)$ and
$z(v_1+0)=\lim_{v\downarrow v_1}z(v)$ is larger than $z(v_2)$. Since $z$ has only finitely many discontinuities, it suffices to decrease or increase $v_1$ somewhat, to a point of continuity, while maintaining $z(v_1)>z(v_2)$.
We have to consider two cases.

A. Assume that $z(v)$ is continuous on $[v_1,v_2]$.
We can then basically argue as in the proof of the only-if part of Proposition \ref{thm1}. Thus, there is a $v_0>0$, $\delta>0$ such that
$z(v)$ is strictly decreasing in $v\in[v_0-\delta,v_0+\delta]$.
We choose $a,b$ such that $xy\in[u_0-\varepsilon,u_0+\varepsilon]$ when $x,y\in[a,b]$.
This is satisfied when $(u_0-\varepsilon)^{1/2}\leq a< b\leq (u_0+\varepsilon)^{1/2}$, and it
guarantees that $T(x)/U(x)$ is strictly increasing in $x\in[a,b]$. Next, we choose $h$
such that $ahx\in[v_0-\delta,v_0+\delta]$ when $x\in[a,b]$, so that $z(ahx)$ is strictly decreasing in $x\in[a,b]$.
For this, we need to take $h$ such that $a^2h\geq v_0-\delta$ and $abh\leq v_0+\delta$.
 This can be done indeed when $a/b\geq(v_0-\delta)/(v_0+\delta$.
Choosing $a$ and $b$ with $a<b$, $a,b\in[(u_0-\varepsilon)^{1/2},(u_0+\varepsilon)^{1/2}]$
such that this latter condition is satisfied, we can apply
 the version of Lemma~\ref{lemma1} with strictly decreasing  $g(x)=z(ahx)$
and  strictly increasing $p(x)/q(x)=T(x)/U(x)$. Thus we get in \eqref{e23} strict inequality $<$ for these $a,b$ and $h$, and this means that $c_{ab}'(h)<0$. This proves Proposition \ref{thm2} for this case.

B. Assume that $z(v)$ has discontinuities on
$[v_1,v_2]$, say at $c_1<c_2<\cdots<c_j$ with $v_1<c_1$ and $v_2>c_j$.
In the case that there is an interval $[v_0-\delta,v_0+\delta]$
contained in one of $(v_1,c_1),(c_1,c_2),\ldots,(c_j,v_2)$ where $z$ is strictly decreasing, we are in the position of case A, and then we are done.
%
%
%
% where we may assume by piecewise smoothness that $c\neq v_1,v_2$ (otherwise, take $v_1$ somewhat smaller or $v_2$ somewhat large while still $h(v_1)>h(v_2)$). By increasing $v_1$ and decreasing $v_2$, if necessary, we may also assume that $h(v)$ is twice continuously differentiable on $v_1\leq v\leq c$ and $c\leq v\leq v_2$ (with left and right derivatives to be considered as $v\uparrow c$ or $v\downarrow c$). In the case that there is an interval $[v_0-\delta,v_0+\delta]$ contained in $(v_1,c)$ or $(c,v_2)$ where $h(v)$ is strictly decreasing, we are in the situation of case A.
Otherwise, we have by F3 that $z'(v)\geq 0$ for all $v\in[v_1,v_2]$, $v\neq c_1,\ldots,c_j$. Then we must have
$z(v_0-0)>z(v_0+0)$ for at least one $v_0=c_1,\ldots,c_j$, for else we would have
$z(v_1)\leq z(c_1-0)\leq z(c_1+0)\leq \cdots \leq z(c_j-0)\leq z(c_j+0)\leq z(v_2)$. For this remaining case, we refer to Appendix \ref{newapp} where we adapt the reasoning in the proof of
Lemma~\ref{lemma1}.

%\section{Important observations}
%We consider the CM model with degree distribution  \cite{Pastor2002,Clauset2009,barabasireview,nsw,Newm03a}
%\begin{equation}\label{degrees}
%P(k)=Ck^{-\tau}\e^{-k/h} \quad {\rm for}\ k\geq 1,
%\end{equation}
%Here, the cutoff $h$ imposes a (soft) limit on the maximum degree, causing the power-law degree distribution to end by a rapid drop.
%The cutoff can be seen as an inherent property of a physical system, for instance due to time constraints (the finite lifetime of a scientist puts a cap on the number of collaborations \cite{Newm01a}), or as an increasing function of the window size in the case of growing networks such as the internet or social networks \cite{barabasireview,nsw}. However, in this paper the cutoff is introduced solely for technical reasons. While the distribution \eqref{degrees} converges to a pure power-law distribution when $h\to\infty$, all moments are finite for any finite $h$, which presents considerable technical advantages over a pure power law, because the generating functions of the degree distribution and other key measures are amenable for analysis. In contrast, for pure power laws, calculations based on the generating function formalism in \cite{nsw} diverge.

\section{Outlook}

For hidden-variable models with scale-free degree distributions and connection probabilities in the F-class, we have shown that the local clustering coefficient $c(h)$ decays with the hidden variable $h$ and that the average clustering coefficient $C(\tau)$ roughly decreases with the tail exponent $\tau$ according to some function that depends on the structural and natural cutoffs. For the typical cutoff choices $\sqrt{N}$ and $N^{1/(\tau-1)}$ this showed that $C$ decays as $N^{2-\tau}\ln N$, confirming an earlier result in \cite{pol2012} and suggesting universal behavior for the entire F-class introduced in this paper. By analyzing the special case of maximally dense graphs, a member of the F-class, we estimated the constant $C(\tau)/N^{2-\tau}\ln N$ and the extremely slow decay that occurs when $\tau\downarrow 2$.

The hidden-variable model is a widely adopted null model, not only because of its ability to generate inhomogeneous random graphs, but also because of recent work that uses the hidden-variable model for constructing geometric versions of random graphs~\cite{krioukov2016clustering}. Here, the common thread is to equip every vertex not only with a weight, but also with a randomly chosen position in some space. Two vertices then form an edge independently with a probability that is proportional to the product of their weights and inversely proportional to some power of their Euclidian distance, that gives rise to a class of connection probabilities that generalizes the F-class with a geometric feature. It would be interesting to use the methods developed in this paper to investigate the clustering in relation to cutoffs and tail exponent in these graphs with an underlying geometry.

Another possible thread is to compare clustering in the hidden-variable model with clustering in other null models, like the configuration model (CM). For any given real-world network, the CM preserves the degree distribution $P(k)$, and makes connections between vertices in the most random way possible \cite{Clauset2009,Newm10a,Pastor2002,email2002,Dhara16}. Given the random nature of the edge assignment, the CM has in principle no degree correlations. But in case of scale-free networks with diverging second moment, this random assignment leads to uncorrelated networks with non-negligible fractions of self-loops (a vertex joined to itself) and multiple connections (two vertices connected by more than one edge). This could be avoided by forbidding self-loops and multiple edges, for instance by generating a sample from the CM and then erasing all the self-loops and multiple edges. This comes at the cost, however, of introducing non-trivial degree correlations among vertices.

In future work we want to investigate the clustering in such erased configuration models. Hidden-variable models are {\it soft} models, where unlike {\it hard} models such as the (erased) CM, graph constraints are satisfied only on average. Soft models are probabilistically more tractable because of the weak dependence structures, which makes it is easier to show model properties. Transferring results for soft models to hard models is not straightforward, and the question is whether the important clustering properties are asymptotically invariant to the soft and hard constraints in the large-network limit. If this is the case, then that would extend the universality class from hidden-variable models in the F-class to a larger class of random network models with single-edge constraints.

\section*{Acknowledgements}
This work is supported by NWO TOP grant 613.001.451 and by the NWO
Gravitation Networks grant 024.002.003.  The work of RvdH is further supported by the NWO
VICI grant 639.033.806.  The work of JvL is further supported by an NWO TOP-GO grant and
by an ERC Starting Grant.

\bibliographystyle{unsrt}
\bibliography{bib}

\appendix

\section{Proof of Proposition~\ref{propCmax}}
Taking the limit $h\downarrow 0$ in~\eqref{mainc}, with $r(u)=u\min(1,1/u)$, we have
\begin{equation}\label{eq:c2}
c_{ab}^{\max}(0)=\frac{\int_{a}^{b}\int_{a}^{b}(xy)^{2-\tau}\min(1,(xy)^{-1})\dd x\dd y }{\left[\int_{a}^{b}x^{1-\tau}\dd x\right]^2},
\end{equation}
where we have written $a$ instead of $ah_{\min}$ for ease of notation.
The denominator in \eqref{eq:c2} is evaluated as
\begin{align}\label{c3}
\Big(\int_a^b x^{1-\tau}\dd x\Big)^2=\frac{1}{(\tau-2)^2}\left(a^{2-\tau}-b^{2-\tau}\right)^2.
\end{align}
For the numerator in \eqref{eq:c2} we compute
\begin{align}\label{c4}
\int_a^b \int_a^b &\min (1, (xy)^{-1})(xy)^{2-\tau}\dd x \dd y \nonumber\\
& = \int_a^{1/b} \int_a^b (xy)^{2-\tau}\dd y \dd x +  \int_{1/b}^b \left(\int_a^{1/x}(xy)^{2-\tau}\dd y +\int_{1/x}^{b}(xy)^{1-\tau}\dd y\right)\dd x \nonumber\\
& = \int_a^{1/b}x^{2-\tau}\dd x \int_a^b y^{2-\tau}\dd y +  \int_{1/b}^bx^{2-\tau} \int_a^{1/x}y^{2-\tau}\dd y\dd x +\int_{1/b}^{b}x^{1-\tau}\int_{1/x}^{b}y^{1-\tau}\dd y\dd x \nonumber\\
& =\frac{(b^{\tau-3}-a^{3-\tau})(b^{3-\tau}-a^{3-\tau})}{(3-\tau)^2} +\frac{1}{3-\tau}\int_{1/b}^bx^{2-\tau} (x^{\tau-3}-a^{3-\tau})\dd x \nonumber\\
&\quad +\frac{1}{2-\tau}\int_{1/b}^{b}x^{1-\tau}(b^{2-\tau}-x^{\tau-2})\dd x \nonumber\\
& =\frac{(b^{\tau-3}-a^{3-\tau})(b^{3-\tau}-a^{3-\tau})}{(3-\tau)^2} +\frac{1}{3-\tau}\left(\ln(b^2)-\frac{a^{3-\tau}(b^{3-\tau}-b^{\tau-3})}{3-\tau}\right)\nonumber\\
&\quad +\frac{1}{2-\tau}\left(\frac{b^{2-\tau}(b^{2-\tau}-b^{\tau-2})}{2-\tau}-\ln(b^2)\right) \nonumber\\
&=\frac{\ln(b^2)}{(3-\tau)(\tau-2)}-\frac{1-b^{2(2-\tau)}}{(2-\tau)^2}+\frac{1-2(ab)^{3-\tau}+a^{2(3-\tau)}}{(3-\tau)^2}.
\end{align}
The last member of~\eqref{c4} equals $I_{ab}^{\max}(\tau)$ in~\eqref{c12}, and the result follows from~\eqref{e2},~\eqref{eq:c2},~\eqref{c3} and~\eqref{c4}.  $\hfill\qedsymbol$

\section{Proof of Proposition~\ref{propCmaxbound}}
Take $u_0\geq 1$ and note that
\begin{align}\label{c15}
f(u)\geq u_0 f(u_0) \min (u_0^{-1},u^{-1}), \ u\geq 0,
\end{align}
since, for $f\in{\rm F}$,
\begin{align}\label{c16}
f(u)\geq f(u_0), \ 0\leq u\leq u_0; \ uf(u)\geq u_0 f(u_0), \ u\geq u_0.
\end{align}
Now for any $c>0$,
\begin{align}\label{c17}
\int_a^b \int_a^b \min (c,(xy)^{-1})(xy)^{2-\tau}\dd x\dd y &=c^{\tau-2}\int_{a\sqrt{c}}^{b\sqrt{c}} \int_{a\sqrt{c}}^{b\sqrt{c}} \min (1,(xy)^{-1})(xy)^{2-\tau}\dd x\dd y \nonumber\\
&=c^{\tau-2}I_{{\rm max};a\sqrt{c},b\sqrt{c}}(\tau).
\end{align}
Also,
\begin{align}\label{c18}
\frac{(\tau-2)^2}{\left(a^{2-\tau}-b^{2-\tau}\right)^2}
=c^{\tau-2}\frac{(\tau-2)^2}{\left((a\sqrt{c})^{2-\tau}-(b\sqrt{c})^{2-\tau}\right)^2}.
\end{align}
The result then follows from combining \eqref{c15}, \eqref{c17} and \eqref{c18}. $\hfill\qedsymbol$

\section{Derivation of Equation (\ref{f9})}\label{sec:leading}
We shall derive~\eqref{f9} assuming~\eqref{assonab} and that $\abs{\ln(ab)/\ln(b^2)}$ is small. In the present case, where $a$ and $b$ are given through~\eqref{def2},~\eqref{def3} and~\eqref{eq:change} with $h_{\min}=1$, this indeed holds since $\abs{\ln(ab)/\ln(b^2)}=(\tau-2)/(3-\tau)$.  
With $s=\tau-2$ we consider
\begin{align}\label{f1}
&C_{ab}^{\rm max}(\tau)
=\frac{s^2}{\left(a^{-s}-b^{-s}\right)^2}\Big[
\frac{\ln (b^2)}{s(1-s)}
-\frac{1-b^{-2s}}{s^2}+\frac{1-2(ab)^{1-s}+a^{2(1-s)}}{(1-s)^2}\Big],
\end{align}
where we have written $a$ instead of $ah_{\min}$ for notational convenience.
We develop, assuming $s\ln (b^2)$ of order unity and less,
\begin{align}\label{f4}
&\frac{\ln (b^2)}{s(1-s)}-\frac{1-b^{-2s}}{s^2}+\frac{1-2(ab)^{1-s}+a^{2(1-s)}}{(1-s)^2}=\tfrac{1}{2}\ln^2 b^2\Big(1-\tfrac13 s \ln (b^2)+\ldots + O((\ln^2 b^2)^{-1})\Big).
\end{align}
Also,
\begin{align}\label{f5}
&\frac{a^{-s}-b^{-s}}{s}=\ln(b/a) \Big[1-\tfrac{1}{2}s\ln(ab)+\tfrac16 s^2 \left(\ln^2 b +\ln b \ln a+\ln^2 a\right)-\ldots \Big].
\end{align}
Note that
\begin{align}\label{f6}
\ln(b/a)=\ln (b^2)-\ln(ab)=\ln (b^2)\Big(1-\frac{\ln(ab)}{\ln (b^2)}\Big),
\end{align}
\begin{align}\label{f7}
\ln^2 b+\ln b \ln a +\ln^2 a=\ln^2 b\Big(1-\frac{\ln(ab)}{\ln b}+\Big(\frac{\ln(ab)}{\ln b}\Big)^2\Big).
\end{align}
Thus we get
\begin{align}\label{f8}
\frac{a^{-s}-b^{-s}}{s}&=\ln (b^2) \Big(1-\frac{\ln(ab)}{\ln (b^2)}\Big)
\Big(1-\tfrac{1}{2}s\ln(ab)+\tfrac16 s^2 \ln^2 b\Big(1+O\Big(\frac{\ln(ab)}{\ln b}\Big)\Big)\Big) \nonumber\\
&\approx \ln (b^2)\left[1-\frac 12 s \ln(ab)+\frac 16 s^2\ln^2 b\right],
\end{align}
where we have used the assumption that  $|\ln(ab)/\ln b|$ is small.

When we insert \eqref{f4} and \eqref{f8} into \eqref{f1} and divide through $\ln^2 (b^2)$, we arrive at \eqref{f9}.

\section{Maximally random graph}\label{sec:pol}
Define Lerch's transcendent
\begin{equation}
\Phi(z,s,v)=\sum_{k=0}^\infty \frac{z^k}{(k+v)^s}.
\end{equation}
In \cite{pol2012} it was shown that for the maximally random graph
\eqref{ex3}
\begin{align}
C_{ab}(\tau)&=   \frac{(\tau-2)^2}{\left(a^{2-\tau}-b^{2-\tau}\right)^2}\Big\{\frac{\pi\ln (b^2)}{\sin \pi(\tau-2)}-\frac{\pi^2\cos\pi(\tau-2) }{\left(\sin \pi(\tau-2)\right)^2}+b^{-2(\tau-2)}\Phi (-b^{-2},2,\tau-2)\nonumber\\ 
&\quad +a^{2(3-\tau)}\Phi \left(-a^{2},2,3-\tau\right)-2(ab)^{3-\tau}\Phi \left(-ab,2,3-\tau\right) \Big\}.\label{newpol}
\end{align}
(The expression is slightly simplified compared to \cite[Eq.~(5)]{pol2012}.) Comparing \eqref{newpol} and \eqref{frontt} shows that the front factor is identical, and that the terms in between brackets differ. Table \ref{tab2} compares the two dominant terms in \eqref{newpol} and \eqref{frontt} and shows that these terms are of comparable magnitude for $\tau-2$ small.
\begin{table}
	\centering
\begin{tabular}{r|rrrr}
$s$ &  $\frac{\pi}{\sin \pi s}$ & $\frac{1}{s(1-s)}$ & $\frac{\pi^2\cos\pi s}{\left(\sin \pi s\right)^2}$ & $\frac{1}{s^2}-\frac{1}{(1-s)^2}$  \\
\hline
0.1 & 10.1664 & 11.1111 & 98.2972 & 98.7654 \\
0.2 &  5.3448 & 6.2500  & 23.1111  & 23.4375\\
0.3 &  3.8832 & 4.7619 & 8.8635 & 9.0703 \\
0.4 & 3.3033 & 4.1666 & 3.3719 & 3.4722 \\
0.5 & 3.1416 & 4.0000 & 0.0000 & 0.0000 \\
\end{tabular}%
\caption{Dominant terms in \eqref{newpol} and \eqref{frontt} for several values of $s=\tau-2$.}
\label{tab2}
\end{table}

\section{Completion proof only-if part Proposition~\ref{thm2}}\label{newapp}
%\chr{MOET NOG WORDEN AANGEPAST}
%\begin{lemma}\label{lemma2}
%Let $P,Q$ be positive and continuous on $[v_1,v_2]$ with $P(x)/Q(x)$ strictly increasing. Then there are $w_1,w_2$ with $v_1\leq w_1<c<w_2\leq v_2$ such that
%\begin{align}\label{a48}
%h_p=\int_{w_1}^{w_2} h(x)p(x)\dd x \geq \int_{w_1}^{w_2} h(x)q(x)\dd x = h_q
%\end{align}
%with
%\begin{align}
%p(x)=\frac{P(x)}{\int_{w_1}^{w_2}P(v)\dd v}, \ q(x)=\frac{Q(x)}{\int_{w_1}^{w_2}Q(v)\dd v}.
%\end{align}
%\end{lemma}
%Indeed, take $w_1$ and $w_2$ such that
%\begin{align}
%|h(x)-h(c-)|&\leq \tfrac18 \Delta, \ w_1\leq x < c,\label{a49}\\
%|h(x)-h(c+)|&\leq \tfrac18 \Delta, \ c< x \leq w_2 \label{a50}
%\end{align}
%and such that
%\begin{align}
%\int_{w_1}^c Q(x) \dd x =\int_c^{w_2} Q(x) \dd x.
%\end{align}
%In \eqref{a49}-\eqref{a50} we have written $\Delta=h(c-)-h(c+)>0$. With $M=\frac12 (h(c-)+h(c+))$, this choice of $w_1$ and $w_2$ guarantees that $h_q\in[M-\frac18 \Delta, M+\frac18 \Delta]$, and so
%\begin{align}
%h(x)-h_q &>\tfrac14 \Delta, \ w_1\leq x <c;\\
%h(x)-h_q &<-\tfrac14 \Delta, \ c< x \leq w_2.
%\end{align}
%When we choose now $R=p(c)/q(c)$, we have
%\begin{align}
%h_p-h_q=\int_{w_1}^{w_2}(h(x)-h_q)(p(x)-R q(x))\dd x<0
%\end{align}
%since the integrand is negative for all $x\in[w_1,w_2]$, $x\neq c$.
%We can now complete the argument as in case $A$ by choosing
%$a,b$ and $h$ with $u_0,\varepsilon$ as before and $[v_0-\delta,v_0+\delta]=[w_1,w_2]$.
We want to find $a,b$ such that
\begin{equation}\label{eq:Tx}
\int_{a}^{b}z(ahx)T(x)\dd x<\int_{a}^{b}z(ahx)U(x)\dd x
\end{equation}
for the case that $z(v)$ has a downward jump at $v=v_0>0$ while being increasing to the left and to the right of $v_0$. Set
\begin{equation}
\Delta = z(v_0-0)-z(v_0+0),\quad M=\frac12\left(z(v_0-0)+z(v_0+0)\right),
\end{equation}
and observe that $M\geq \frac12\Delta>0$ since $z(v)\geq 0$ for all $v$. We can find $\delta>0$ such that
\begin{align}
	z(v_0-0)&\geq z(v)\geq z(v_0-0)-\frac 18 \Delta,\quad v_0-\delta\leq v<v_0\label{eq:k-}\\
	z(v_0+0)&\leq z(v)\leq z(v_0+0)+\frac18\Delta, \quad v_0<v\leq v_0+\delta. \label{eq:k+}
\end{align}
Next, let
\begin{equation}\label{eq:lv}
l(v)=f(v)v^{1-\tau},\quad v>0,
\end{equation}
and observe that $l(v)$ is positive and continuous at $v=v_0$. Hence, we can choose $\delta>0$ such that, in addition to~\eqref{eq:k-} and~\eqref{eq:k+},
\begin{equation}\label{eq:abs}
\abs{\frac{l(v)}{l(v_0)}-1}\leq \lambda,\quad v\in[v_0-\delta,v_0+\delta],
\end{equation}
where $\lambda$ is any number between 0 and $\tfrac{5}{16}\Delta/(2M+\tfrac{7}{16}\Delta)$. As in case A of the proof, we choose $a,b$ and $h$ such that
\begin{equation}
xy \in [u_0-\varepsilon,u_0+\varepsilon],
\end{equation}
when $x,y\in[a,b]$ and
\begin{equation}\label{eq:ahx}
ahx\in[v_0-\delta,v_0+\delta],
\end{equation}
when $x\in[a,b]$. Thus, we let $(u_0-\varepsilon)^{1/2}\leq a<b\leq (u_0+\varepsilon)^{1/2}$ such that $1>a/b\geq (v_0-\delta)/(v_0+\delta)$. Below, we shall transform the two integrals by the substitution $v=ah_0x$ for a special choice of $h=h_0$ to an integral over an interval $[w_1,w_2]$ having $v_0$ as midpoint. This $h_0$ is given by
\begin{equation}
h_0=\frac{2v_0}{a^2+ab}\in\left[\frac{v_0-\delta}{a^2},\frac{v_0+\delta}{ab}\right].
\end{equation}
Indeed, this $h_0$ satisfies~\eqref{eq:ahx} since
\begin{equation}
\begin{aligned}[b]
\frac{2v_0}{a^2+ab}\leq \frac{v_0+\delta}{ab}&\iff 2bv_0\leq (b+a)(v_0+\delta)\\
&\iff (1-\frac ab)v_0\leq (1+\frac ab)\delta \\
&\iff \frac ab\geq \frac{v_0-\delta}{v_0+\delta},
\end{aligned}
\end{equation}
and
\begin{equation}
\begin{aligned}[b]
\frac{2v_0}{a^2+ab}\geq \frac{v_0-\delta}{ab}&\iff 2av_0\leq (b+a)(v_0-\delta)\\
& \iff (\frac ab - 1)v_0\geq- (1+\frac ab)\delta \\
&\iff \frac ab\geq \frac{v_0-\delta}{v_0+\delta}.
\end{aligned}
\end{equation}
In the integrals in the inequality in~\eqref{eq:Tx} with $h=h_0$, we substitute $ah_0x=v$, and the inequality to be proved becomes
\begin{equation}
z_t:=\int_{w_1}^{w_2}z(v)t(v)\dd v < \int_{w_1}^{w_2}z(v)u(v)\dd v =:z_u.
\end{equation}
Here
\begin{equation}
w_1=a^2h_0,\quad w_2=abh_0
\end{equation}
so that $v_0=\tfrac12(a^2+ab)h_0$ is the midpoint of the integration interval $[w_1,w_2]\subset[v_0-\delta,v_0+\delta]$, and $t(v)$ and $u(v)$ are the pdf's
\begin{equation}\label{eq:tv}
t(v)=\frac{1}{ah_0}T\left(\frac{v}{ah_0}\right),\quad u(v)=\frac{1}{ah_0}U\left(\frac{v}{ah_0}\right)
\end{equation}
for which $t(v)/u(v)$ is strictly increasing in $v\in[w_1,w_2]$. We shall show that $z_u\in(M-\tfrac38\Delta,M+\tfrac38\Delta)$, and so, by~\eqref{eq:k-} and~\eqref{eq:k+},
\begin{equation}
z(v)-z_u>0,\quad w_1\leq v<v_0;\quad z(v)-z_u<0, v_0<v\leq w_2.
\end{equation}
With $R=t(v_0)/u(v_0)$, this implies that
\begin{equation}
z_t-z_u=\int_{w_1}^{w_2}(z(v)-z_u)(t(v)-Ru(v))\dd v<0,
\end{equation}
since the integrand is negative for all $v\neq v_0$.

To show that $z_u\in(M-\tfrac38\Delta,M+\tfrac38\Delta)$, we note that the pdf $u(v)$ is built from the function $l(v)$ in~\eqref{eq:lv} via~\eqref{e11} and~\eqref{eq:tv}. In terms of this $l(v)$ we can write $z_u$ as
\begin{equation}\label{eq:ku}
z_u=\frac{\int_{w_1}^{w_2}z(v)l(v)\dd v}{\int_{w_1}^{w_2}l(v)\dd v}.
\end{equation}
Now, by~\eqref{eq:abs},
\begin{equation}\label{eq:zwlow}
(w_2-w_1)l(v_0)(1-\lambda)\leq \int_{w_1}^{w_2}l(v)\dd v\leq (w_2-w_1)(1+\lambda)l(v_0).
\end{equation}
Also, by~\eqref{eq:k-},~\eqref{eq:k+} and~\eqref{eq:abs} and the fact that $v_0=\frac{1}{2}(w_1+w_2)$,
\begin{equation}\label{eq:intw}
\begin{aligned}[b]
\int_{w_1}^{w_2}z(v)l(v)\dd v &= \int_{w}^{v_0}z(v)l(v)\dd v +\int_{v_0}^{z}z(v)l(v)\dd v\\
&\leq z(v_0-0)\int_{w}^{v_0}l(v)\dd v +z(v_0+\frac 18 \Delta)\int_{v_0}^{z}l(v)\dd v\\
&\leq \frac 12 (w_2-w_1)(1+\lambda) l(v_0)\left[z(v_0-0)+z(v_0+0)+\frac18\Delta\right]\\
&=  (w_2-w_1)(1+\lambda) l(v_0)\left[M+\frac{1}{16}\Delta\right],
\end{aligned}
\end{equation}
and in a similar fashion
\begin{equation}\label{eq:kvlv}
\int_{w_1}^{w_2}z(v)l(v)\dd v \geq (w_2-w_1)(1-\lambda) l(v_0)\left[M-\frac{1}{16}\Delta\right].
\end{equation}
From~\eqref{eq:zwlow},~\eqref{eq:intw} and~\eqref{eq:kvlv} we then get
\begin{equation}
\frac{1-\lambda}{1+\lambda}(M-\frac{1}{16}\Delta)\leq z_u\leq \frac{1+\lambda}{1-\lambda}(M+\frac{1}{16}\Delta).
\end{equation}
Now
\begin{align}
	&\frac{1+\lambda}{1-\lambda}(M+\frac{1}{16}\Delta)<M+\frac{3}{8}\Delta \iff \lambda<\frac{\frac{5}{16}\Delta}{2M+\frac{7}{16}\Delta},\nonumber\\
	&\frac{1-\lambda}{1+\lambda}(M-\frac{1}{16}\Delta)<M+\frac{3}{8}\Delta \iff \lambda<\frac{\frac{5}{16}\Delta}{2M-\frac{7}{16}\Delta}.\nonumber
	\end{align}
	Then it follows from the choice of $\lambda$ in~\eqref{eq:abs} that $z_u\in(M-\tfrac38\Delta,M+\tfrac38\Delta)$ for such $\lambda$.

\section{Monotonicity properties for $C_{ab}(\tau)$}\label{sec:monotone}
In this appendix we show that $C_{ab}(\tau)$ is bounded from above by a closely related function that decreases in $\tau$. Notice that Proposition \ref{thm1} assumes $a$ and $b$ fixed. We have from~\eqref{mainc} and~\eqref{def2},~\eqref{def3} that
\begin{align}
a&=a(\tau) =\left(\frac{1}{N\mean{h}}\right)^{1/2}, \quad 
b= b(\tau) =(N\mean{h})^{\frac{3-\tau}{2(\tau-1)}} %\left(\frac{\tau-2}{\tau-1}\right)^{1/2}\left(h_{\min}N\right)^{-1/2}N^{\frac{1}{\tau-1}}\left(\frac{h_{\min}^{2-\tau}-N^{2-\tau}}{h_{\min}^{1-\tau}-N^{1-\tau}}\right)^{-1/2}\label{eq:btau},
\end{align}
where $\mean{h}$ is given in~\eqref{eq:meanh}.
Below we shall use that $\mean{h}$ decreases in $\tau\in(2,3]$; this is clear intuitively and can be proved rigorously by using Lemma~\ref{lemma1}.
We have
\begin{equation}\label{eq:Cab}
C_{ab}(\tau) = A(\tau)G(\tau,a(\tau),b(\tau)),
\end{equation}
where
\begin{equation}
A(\tau) = \int_{h_{\min}}^{N}\rho(h)(1-(1+h)\e^{-h})\dd h,
\end{equation}
with density $\rho(h)=Ch^{-\tau}$ on $[h_{\min},N]$ 
%\begin{equation}
%\rho(h) = \frac{h^{-\tau}}{\int_{h_{\min}}^{N}h_1^{-\tau}\dd h_1},
%\end{equation}
and
\begin{equation}
G(\tau,a,b) = \frac{\int_{ah_{\min}}^{b}\int_{ah_{\min}}^{b}(xy)^{2-\tau}f(xy)\dd x \dd y}{\left(\int_{ah_{\min}}^{b}x^{1-\tau}\dd x\right)^2}.
\end{equation}
Proposition~\ref{thm1} says that
\begin{enumerate}
	\item[(i)] $G$ decreases in $\tau$ (when $a$ and $b$ are fixed).
\end{enumerate}
With the method of the proof of Proposition~\ref{thm1} in Section~\ref{sec:proof}, we will show that
\begin{enumerate}
	\item[(ii)] $A$ decreases in $\tau$,
	\item[(iii)] $G$ increases in $a$ and in $b$.
\end{enumerate}
Showing that $G(\tau,a(\tau),b(\tau))$ decreases in $\tau$ is complicated by the facts that $a(\tau)$ increases in $\tau$, see(iii), and that the dependence of $a(\tau)$ and $b(\tau)$ on $\tau$ is rather involved. Let $m$ and $M$ be the minimum and maximum, respectively, of $\mean{h}$ when $\tau\in[2,3]$ ($m=\mean{h}\mid_{h=3}\approx 2h_{\min}$, $M=\mean{h}\mid_{\tau\downarrow 2}\approx h_{\min}\ln(N/h_{\min})$ from~\eqref{eq:meanh} and the monotonicity of $\mean{h}$).
Letting
\begin{equation}
\bar{a}:=\left(N m \right)^{-1/2},\quad \bar{b}(\tau) = \left(NM \right)^{\frac{3-\tau}{2(\tau-1)}},
\end{equation}
we have $a(\tau)\leq \bar{a}$, $b(\tau)\leq \bar{b}(\tau)$, and so by (iii)
\begin{equation}\label{eq:Flow}
G(\tau,a(\tau),b(\tau))\leq G(\tau,\bar{a},\bar{b}(\tau)).
\end{equation}
The right-hand side of~\eqref{eq:Flow} decreases in $\tau$ by (i) and (iii) and the fact that $\bar{b}(\tau)$ decreases in $\tau$. Therefore, $C_{ab}(\tau)$ in~\eqref{eq:Cab} is bounded above by a closely related function that does decrease in $\tau$.

%We let $\tau\in(2,3), 0<a<b<\infty$ and we consider
%\begin{equation}\label{eq:F2}
%G(\tau,a,b) = \frac{\int_{ah_{\min}}^{b}\int_{ah_{\min}}^{b}(xy)^{2-\tau}f(xy)\dd x \dd y}{\left(\int_{ah_{\min}}^{b}x^{1-\tau}\dd x\right)^2},
%\end{equation}
%where $F$ is assumed to satisfy F1-F4. Also, for $l>h_{\min}$, with $h_{\min}>0$ fixed, consider
%\begin{equation}
%A(\tau) = \frac{\int_{h_{\min}}^{N}h^{-\tau}(1-(1+h)\e^{-h})\dd h}{\int_{h_{\min}}^{N}h^{-\tau}\dd h}.
%\end{equation}
We have shown that $G(\tau,a,b)$ decreases in $\tau\in[2,3]$. We shall show now that
\begin{enumerate}
	\item [(ii)] $A$ decreases in $\tau$ and increases in $l$,
	\item [(iii)] $G$ increases in $a$ and in $b$.
\end{enumerate}

\begin{proof}[Proof that $A$ decreases in $\tau$]
	Since $\frac{\dd}{\dd \tau}h^{-\tau}= -h^{-\tau}\ln h$, we have
	\begin{equation}\label{eq:At}
	\begin{aligned}[b]
	\frac{\partial A}{\partial \tau}\leq 0 	&\iff \int_{h_{\min}}^{N}- h^{-\tau} (1-(1+h)\e^{-h})\ln h\dd h \int_{h_{\min}}^{N}h^{-\tau}\dd h\\
&\quad - \int_{h_{\min}}^{N}h^{-\tau} (1-(1+h)\e^{-h})\dd h  \int_{h_{\min}}^{N}- h h^{-\tau}\ln h\dd h\leq 0\\
&\iff \frac{\int_{h_{\min}}^{N} h^{-\tau} (1-(1+h)\e^{-h})\ln h\dd h }{\int_{h_{\min}}^{N}h^{-\tau} (1-(1+h)\e^{-h})\dd h }  \geq \frac{\int_{h_{\min}}^{N}h^{-\tau}\ln h\dd h}{\int_{h_{\min}}^{N} h^{-\tau}\dd h}.
	\end{aligned}
	\end{equation}
	Consider on $[h_{\min},N]$ the pdf's
	\begin{align}
	p(h) &= \frac{h^{-\tau} (1-(1+h)\e^{-h})}{\int_{h_{\min}}^{N}h_1^{-\tau} (1-(1+h_1)\e^{-h_1})\dd h_1 },\\
	 q(h) &= \frac{ h^{-\tau}}{\int_{h_{\min}}^{N} h_1^{-\tau}\dd h_1}=\rho(h).
	\end{align}
	Clearly $p(h)/q(h) = C(1-(1+h)\e^{-h})$ with $C$ independent of $h\in [h_{\min},N]$. Hence, $p(h)/q(h)$ is increasing in $[h_{\min},N]$. Also, $\ln(h)=g(h)$ is increasing in $[h_{\min},N]$. Hence, by Lemma~\ref{lemma1},
	\begin{equation}
	\int_{h_{\min}}^{N}g(h)p(h)\dd h \geq \int_{h_{\min}}^{N}g(h)q(h)\dd h,
	\end{equation}
	and this is the last inequality in~\eqref{eq:At}.
\end{proof}

%\begin{proof}[Proof that $A$ increases in $l$]
%	We have
%	\begin{equation}\label{eq:Al}
%	\begin{aligned}[b]
%		\frac{\partial A}{\partial l}&\geq 0 \iff l^{-\tau}(1-(1+l)\e^{-l}) \int_{h_{\min}}^{N}h^{-\tau}\dd h -l^{-\tau} \int_{h_{\min}}^{N}h^{-\tau} (1-(1+h)\e^{-h})\dd h\geq 0.\\
%		\end{aligned}
%	\end{equation}
%	Since $1-(1+h)\e^{-h}$ is increasing, we have
%	\begin{align}
%	&l^{-\tau}(1-(1+l)\e^{-l}) \int_{h_{\min}}^{N}h^{-\tau}\dd h\geq l^{-\tau} \int_{h_{\min}}^{N}h^{-\tau} (1-(1+h)\e^{-h})\dd h,
%	\end{align}
%	and this gives the second inequality in~\eqref{eq:Al}.
%\end{proof}

\begin{proof}[Proof that $G$ increases in $b$]
	Again, for notational simplicity, we write $a$ and $b$ instead of $ah_{\min}$ and $bh_{\min}$ respectively.
	Let $\tau$ and $a$ be fixed, and set
	\begin{equation}
	p(x,y)=(xy)^{2-\tau}f(xy), \quad P(b,y)=\int_{a}^{b}p(x,y)\dd x.
	\end{equation}
	We have
\begin{equation}
\begin{aligned}[b]
\frac{\dd }{\dd b}\left[\int_{a}^{b}\int_{a}^{b}(xy)^{2-\tau}f(xy)\dd x \dd y\right]&=\frac{\dd }{\dd b}\left[\int_{a}^{b}P(b,y)\dd y\right] = P(b,b) +\int_{a}^{b}\frac{\partial P}{\partial b}(b,y)\dd y\\
&=\int_{a}^{b}p(x,b)\dd x + \int_{a}^{b}p(b,y)\dd y = 2\int_{a}^{b}p(x,b)\dd x\\
& = 2\int_{a}^{b}(xb)^{2-\tau}f(xb)\dd x
\end{aligned}
\end{equation}
because of symmetry of $p(x,y)$. Also,
\begin{equation}
\frac{\dd }{\dd b}\left(\int_{a}^{b}x^{1-\tau}\dd x\right)^2 = 2b^{1-\tau}\int_{a}^{b}x^{1-\tau}\dd x.
\end{equation}
Therefore,
\begin{equation}\label{eq:Fb}
\begin{aligned}[b]
	\frac{\partial G}{\partial b}\geq 0
	&\iff  2\int_{a}^{b}(xb)^{2-\tau}f(xb)\dd x\left(\int_{a}^{b}x^{1-\tau}\dd x\right)^2\\
	&\quad \quad - 2\int_{a}^{b}\int_{a}^{b}(xy)^{2-\tau}f(xy)\dd x \dd yb^{1-\tau}b^{1-\tau}\int_{a}^{b}x^{1-\tau}\dd x\geq 0\\
	&\iff \frac{\int_{a}^{b}(xb)^{2-\tau}f(xb)\dd x }{\int_{a}^{b}\int_{a}^{b}(xy)^{2-\tau}f(xy)\dd x \dd y } \geq \frac{b^{1-\tau}}{\int_{a}^{b}x^{1-\tau}\dd x}\\
	&\iff W(b)\geq V(b),
\end{aligned}
\end{equation}
where $W(x)$ and $V(x)$ are the pdf's as defined in~\eqref{e6} and~\eqref{e7}. Since $W(x)/V(x)$ increases in $x\in [a,b]$, we get
\begin{equation}
\begin{aligned}[b]
1& = \int_{a}^{b}W(x)\dd x= \int_{a}^{b}\frac{W(x)}{V(x)}V(x)\dd x \leq \frac{W(b)}{V(b)}\int_{a}^{b}V(x)\dd x = \frac{W(b)}{V(b)},
\end{aligned}
\end{equation}
and this gives the last inequality in~\eqref{eq:Fb}.
\end{proof}

\begin{proof}[Proof that $G$ increases in $a$]
	This proof is very similar to the proof that $G$ increases in $b$. Let $\tau$ and $b$ be fixed. We now have
	\begin{equation}
	\begin{aligned}[b]
	\frac{\dd }{\dd a}&\left[\int_{a}^{b}\int_{a}^{b}(xy)^{2-\tau}f(xy)\dd x \dd y\right]=-2 \int_{a}^{b}(xa)^{2-\tau}f(xa)\dd x,
	\end{aligned}
	\end{equation}
	and
	\begin{equation}
	\frac{\dd }{\dd a}\left(\int_{a}^{b}x^{1-\tau}\dd x\right)^2 = -2a^{1-\tau}\int_{a}^{b}x^{1-\tau}\dd x.
	\end{equation}
	Then, as in~\eqref{eq:Fb},
	\begin{equation}
	\begin{aligned}[b]
		\frac{\partial G}{\partial a}\geq 0 &\iff \frac{\int_{a}^{b}(xa)^{2-\tau}f(xa)\dd x }{\int_{a}^{b}\int_{a}^{b}(xy)^{2-\tau}f(xy)\dd x \dd y } \leq \frac{a^{1-\tau}}{\int_{a}^{b}x^{1-\tau}\dd x}\\
		&\iff W(a)\leq V(a),
	\end{aligned}
	\end{equation}
	and the inequality follows again from the increasingness of $W(x)/V(x)$.
\end{proof}

\section{Choice of the natural cutoff}\label{sec:natcut}
We consider random variables $\underline{h}_i$, $i=1,\dots,N$, that are i.i.d. with density $p(h)=Ch^{-\tau}$, $h_{\min}\leq h <\infty$. We have
\begin{equation}\label{eq:x1}
\mathbb{E}[\max_i\underline{h}_i]=h_{\min}\Gamma\left(\frac{\tau-2}{\tau-1}\right)\frac{\Gamma(N+1)}{\Gamma\left(N+\frac{\tau-2}{\tau-1}\right)}\approx h_{\min}\Gamma\left(\frac{\tau-2}{\tau-1}\right)N^{\frac{1}{\tau-1}}.
\end{equation}
The first identity in~\eqref{eq:x1} is exact, and follows on elaborating
\begin{equation}
\begin{aligned}[b]
\mathbb{E}[\max_i\underline{h}_i]&=\int_{h_{\min}}^{\infty}h\dd\left[\mathbb{P}^N(\underline{h}\leq h)\right]\\
&=\int_{h_{\min}}^{\infty}h\dd\left[(1-\mathbb{P}(\underline{h}> h))^N\right]= \int_{h_{\min}}^{\infty}h\dd\left[\left(1-\left(\frac{h}{h_{\min}}\right)^{1-\tau}\right)^N\right],
\end{aligned}
\end{equation}
using the substitution $t=(h/h_{\min})^{1-\tau}\in(0,1]$ and the expression of the $B$-function in terms of the $\Gamma$ function. The approximate identity in~\eqref{eq:x1} follows from $\Gamma(n+a)/\Gamma(n+b)\approx n^{a-b}$, which is quite accurate when $a,b\in[0,1]$ and $n$ large. 

Next, there is the inequality
\begin{equation}\label{eq:x3}
\left(\frac{\tau-1}{\tau-2}\right)^{\frac{1}{\tau-1}}\leq \Gamma\left(\frac{\tau-2}{\tau-1}\right)\leq \frac43 \left(\frac{\tau-1}{\tau-2}\right)^{\frac{1}{\tau-1}}, \quad 2\leq\tau\leq 3.
\end{equation}
This inequality follows from
\begin{equation}\label{eq:x4}
u^u\leq \Gamma(1+u)=u\Gamma(u)\leq \frac43 u^u,
\end{equation}
with $u=(\tau-2)/(\tau-1)\in(0,\tfrac12]$ that can be shown by considering the concave function $\ln(\Gamma(1+u))-u\ln u$ which vanishes at $u=0,1$ and is positive at $u=\tfrac{1}{2}$ (the upper bound in~\eqref{eq:x4} follows from a numerical inspection of this function). 

Then from~\eqref{eq:x3} we get, using $\mean{h}=h_{\min}(\tau-1)/(\tau-2)$, 
\begin{equation}
h_{\min}^{\frac{\tau-2}{\tau-1}}\left(N\mean{h}\right)^{\frac{1}{\tau-1}}\leq h_{\min}\Gamma\left(\frac{\tau-2}{\tau-1}\right)N^{\frac{1}{\tau-1}}\leq\frac43  h_{\min}^{\frac{\tau-2}{\tau-1}}\left(N\mean{h}\right)^{\frac{1}{\tau-1}},
\end{equation}
showing that the order of magnitude of $\mathbb{E}[\max_i\underline{h}_i]$ is $(N\mean{h})^{1/(\tau-1)}$. This motivates our choice of $h_c$ in~\eqref{def3}.

\end{document}